\documentclass[11pt]{article}

\usepackage[letterpaper,margin=1.1in]{geometry}
\usepackage[T1]{fontenc}
\usepackage{amsmath,amsthm,amssymb}
\usepackage{newpxtext}
\usepackage{newpxmath}
\usepackage{mathtools}
\usepackage{microtype}
\usepackage[hidelinks]{hyperref}
\usepackage[nameinlink,noabbrev]{cleveref}

\numberwithin{equation}{section}
\allowdisplaybreaks

\newtheorem{theorem}{Theorem}[section]
\newtheorem{proposition}[theorem]{Proposition}

\newtheorem{corollary}[theorem]{Corollary}
\theoremstyle{definition}

\newtheorem{example}[theorem]{Example}
\theoremstyle{remark}
\newtheorem{remark}[theorem]{Remark}

\crefname{theorem}{theorem}{theorems}
\crefname{proposition}{proposition}{propositions}
\crefname{lemma}{lemma}{lemmas}
\crefname{corollary}{corollary}{corollaries}
\crefname{definition}{definition}{definitions}
\crefname{example}{example}{examples}
\crefname{remark}{remark}{remarks}

\newcommand{\SP}{\mathrm{SP}}
\newcommand{\OP}{\mathrm{OP}}

\newcommand{\vac}{\mathbf{1}}
\newcommand{\cQ}{\mathcal{Q}}
\newcommand{\Yop}{\mathcal{Y}}
\newcommand{\Ht}{\mathsf{H}}

\newcommand{\Pf}{\operatorname{Pf}}

\newcommand{\K}{\mathbb{K}}
\newcommand{\Gam}{\Gamma}
\newcommand{\rhoop}[1]{\rho_{#1}}

\newcommand{\qbinom}[3]{\genfrac{[}{]}{0pt}{}{#1}{#2}_{#3}}

\title{Mixed Products of Modified Greaves--Jing--Zhu Operators}
\author{S.-J. Lee}
\date{}

\begin{document}

\maketitle

\begin{abstract}
Let $\Yop(z;t)$ be the modified Greaves--Jing--Zhu 
operator on the odd power-sum ring. We first point out that this
operator can be obtained from the classical neutral operator by a
simple diagonal change of variables. We then study products in which
the two deformation parameters are not necessarily the same. For two
parameters $t$ and $s$, we compute the scalar factor that appears in
the mixed product. This factor has an explicit exponential form and,
in a completed setting, can also be written as a quotient of infinite
$t$-Pochhammer products. We also give a recurrence for its
coefficients, a product formula for several mixed operators, and
formulas for the coefficients obtained after applying the operators to
$\vac$.

A particularly simple case occurs when $s=t^M$. In this case the
scalar factor becomes the finite quotient
$(u;t)_M/(-u;t)_M$. Its coefficients are signed principal
specializations of one-row Schur $Q$-functions. As a result, after
removing the signs, these coefficients are nonnegative palindromic
polynomials. We also give a Gaussian-binomial formula and a
finite-order recurrence. 
\end{abstract}

\section{Introduction}

Vertex operators provide a useful way to organize identities involving
symmetric functions, Heisenberg algebras, and free-fermion
representations. Greaves, Jing, and Zhu introduced a charged-fermion
operator construction for the $t$-Schur functions and the associated
$t$-Schur measure \cite{GreavesJingZhu}. We adapted the same
idea to the odd power-sum ring, which is the standard setting for
Schur $Q$-functions \cite{Lee}. The resulting modified operator has
the same scalar factor as the classical neutral operator when both
operators use the same parameter. Its modes generate shifted
$t$-Schur functions and lead to a two-row formula, a Pfaffian
Giambelli identity, a Cauchy identity, and a finite shifted Gessel
formula.

The goal of this paper is more limited. We only study what happens
when two modified operators with different parameters are multiplied.
 General linear transformations
of vertex-operator presentations have been studied in a broader
setting; see, for example, \cite{Rozhkovskaya}. In the present case,
the diagonal form of the modified Greaves--Jing--Zhu operator makes
the mixed factor especially explicit.

Let
\[
  \rho_t(p_n)=(1-t^n)p_n
  \qquad(n\geq1\text{ odd}).
\]
Our first observation is
\[
  \Yop(z;t)=\rho_t\Phi(z)\rho_t^{-1},
  \qquad
  \cQ_\lambda(X;t)=Q_\lambda[X-tX],
\]
where $\Phi(z)$ is the classical neutral operator. For two independent
parameters, the mixed product has the form
\begin{equation}\label{eq:intro1-mixed}
 \Yop(z;t)\Yop(w;s)
 =F_{t,s}(w/z):\Yop(z;t)\Yop(w;s):,
\end{equation}
where
\begin{equation}\label{eq:intro1-F}
 F_{t,s}(u)
 =\exp\!\left(
   -2\sum_{\substack{n\geq1\\n\text{ odd}}}
    \frac{1-s^n}{1-t^n}\frac{u^n}{n}
  \right).
\end{equation}
The specialization $s=t$ gives $(1-u)/(1+u)$, which is the usual
neutral factor. In the $t$-adic completion, or analytically for
$|t|<1$, this mixed factor may also be written as
\[
 F_{t,s}(u)
 =\frac{(u;t)_\infty(-su;t)_\infty}
        {(-u;t)_\infty(su;t)_\infty}.
\]
We also derive a recurrence for the coefficients of $F_{t,s}(u)$ and
use the mixed product formula to compute coefficients after applying
the operators to $\vac$.

The most concrete case is
\[
  s=t^M,
  \qquad M\in\mathbb Z_{>0}.
\]
Then the infinite product reduces to the finite product
\begin{equation}\label{eq:intro1-cyclo}
 F_{t,t^M}(u)
 =\prod_{j=0}^{M-1}\frac{1-t^ju}{1+t^ju}
 =\frac{(u;t)_M}{(-u;t)_M}.
\end{equation}
If
\[
 F_{t,t^M}(u)=\sum_{m\geq0}f_m^{[M]}(t)u^m,
\]
then
\begin{equation}\label{eq:intro1-principal}
 f_m^{[M]}(t)
 =(-1)^m q_m(1,t,\ldots,t^{M-1}).
\end{equation}
Therefore $(-1)^mf_m^{[M]}(t)$ is a nonnegative palindromic
polynomial. We give an explicit Gaussian-binomial formula and a
finite recurrence of order $M$.

The paper is organized as follows. In
\cref{sec:preliminaries} we review the odd power-sum ring and prove
the diagonal-conjugation and plethystic formulas. In \cref{sec:mixed}
we prove the general mixed product formula, its infinite-product
form, its coefficient recurrence, and the corresponding coefficient
formulas after applying the operators to $\vac$. In
\cref{sec:cyclotomic-mixed} we specialize to $s=t^M$ and study the
finite product and its coefficients.

\section{The odd power-sum ring and diagonal conjugation}
\label{sec:preliminaries}

\subsection{Strict partitions and Schur \texorpdfstring{$Q$}{Q}-functions}

A \emph{strict partition} is a finite sequence
$\lambda=(\lambda_1>\cdots>\lambda_\ell>0)$. We write
$|\lambda|=\lambda_1+\cdots+\lambda_\ell$ and
$\ell(\lambda)=\ell$. Let $\SP(N)$ denote the set of strict
partitions of $N$. Let $\OP(N)$ denote the set of partitions of $N$
all of whose parts are odd. Euler's identity gives
$|\SP(N)|=|\OP(N)|$.

Let
\[
 \Gam=\mathbb{Q}[p_1,p_3,p_5,\ldots]
\]
be the odd power-sum ring, graded by $\deg p_n=n$. For a partition
$\alpha=(1^{m_1}3^{m_3}\cdots)\in\OP(N)$, set
\[
 p_\alpha=\prod_i p_{\alpha_i},
 \qquad
 z_\alpha=\prod_{r\geq1}r^{m_r}m_r!.
\]
The elements $p_\alpha$, $\alpha\in\OP(N)$, form a basis of the
homogeneous component $\Gam_N$.

We use the standard Schur $Q$ scalar product
\begin{equation}\label{eq:Q-inner-product}
 \langle p_\alpha,p_\beta\rangle
 =2^{-\ell(\alpha)}z_\alpha\,\delta_{\alpha\beta}.
\end{equation}
The Schur $Q$-functions $Q_\lambda$, $\lambda\in\SP(N)$, form another
basis of $\Gam_N$. We set
\[
 P_\lambda=2^{-\ell(\lambda)}Q_\lambda.
\]
Then
\begin{equation}\label{eq:PQ-duality}
 \langle Q_\lambda,P_\mu\rangle=\delta_{\lambda\mu},
 \qquad
 \langle Q_\lambda,Q_\mu\rangle
 =2^{\ell(\lambda)}\delta_{\lambda\mu}.
\end{equation}
We follow the standard conventions in \cite[Chapter III, Section
8]{Macdonald}; see also \cite{JingSpin,Stembridge}.

The one-row functions $q_r=Q_{(r)}$ are determined by
\begin{equation}\label{eq:one-row-generating}
 Q_X(z)
 :=\sum_{r\geq0}q_r(X)z^r
 =\exp\!\left(
    2\sum_{\substack{n\geq1\\n\text{ odd}}}
      \frac{p_n(X)}{n}z^n
   \right)
 =\prod_i\frac{1+x_i z}{1-x_i z}.
\end{equation}
Here $q_0=1$ and $q_r=0$ for $r<0$. For $r\geq s\geq0$, define
\begin{equation}\label{eq:two-row-classical}
 Q_{(r,s)}
 =q_rq_s+2\sum_{k=1}^{s}(-1)^kq_{r+k}q_{s-k}.
\end{equation}
In particular, $Q_{(r,0)}=q_r$ and $Q_{(r,r)}=0$. If $\lambda$ has
odd length, append a zero part. The Pfaffian Giambelli identity is
\begin{equation}\label{eq:Q-Pfaffian}
 Q_\lambda
 =\Pf\bigl(Q_{(\lambda_i,\lambda_j)}\bigr).
\end{equation}

For later use, define coefficients $\chi_\alpha^\lambda$ by the
power-sum expansion
\begin{equation}\label{eq:spin-expansion}
 Q_\lambda
 =\sum_{\alpha\in\OP(|\lambda|)}
   \frac{2^{\ell(\alpha)}}{z_\alpha}
   \chi_\alpha^\lambda p_\alpha.
\end{equation}
With the usual normalization, these are the spin-character
coefficients. Their representation-theoretic interpretation will not
be used below; we only need \eqref{eq:spin-expansion} and the
orthogonality implied by \eqref{eq:Q-inner-product}.

\subsection{The modified neutral operator}

Let $t$ be an indeterminate and work over $\mathbb{Q}(t)$. For positive
odd $n$, set
\begin{equation}\label{eq:Heisenberg}
 a_{-n}=p_n,
 \qquad
 a_n=\frac{n}{2}\frac{\partial}{\partial p_n}.
\end{equation}
Then
\begin{equation}\label{eq:Heisenberg-commutator}
 [a_m,a_n]=\frac{m}{2}\delta_{m,-n}
 \qquad(m,n\text{ odd}).
\end{equation}
The classical neutral operator is
\begin{equation}\label{eq:Phi}
 \Phi(z)
 =\exp\!\left(
   \sum_{\substack{n\geq1\\n\text{ odd}}}
    \frac{2}{n}a_{-n}z^n
  \right)
  \exp\!\left(
   -\sum_{\substack{n\geq1\\n\text{ odd}}}
    \frac{2}{n}a_nz^{-n}
  \right).
\end{equation}
Its Fourier expansion is
\[
 \Phi(z)=\sum_{m\in\mathbb Z}\phi_m z^{-m}.
\]

Following \cite{Lee}, define the modified operator
\begin{equation}\label{eq:Yt}
 \Yop(z;t)
 =\exp\!\left(
   \sum_{\substack{n\geq1\\n\text{ odd}}}
    \frac{2(1-t^n)}{n}a_{-n}z^n
  \right)
  \exp\!\left(
   -\sum_{\substack{n\geq1\\n\text{ odd}}}
    \frac{2}{n(1-t^n)}a_nz^{-n}
  \right).
\end{equation}
In terms of derivatives, the second exponential equals
\[
 \exp\!\left(
  -\sum_{\substack{n\geq1\\n\text{ odd}}}
    \frac{1}{1-t^n}
    \frac{\partial}{\partial p_n}z^{-n}
 \right).
\]
Write
\begin{equation}\label{eq:Y-modes}
 \Yop(z;t)=\sum_{m\in\mathbb Z}\Yop_m(t)z^{-m}.
\end{equation}
For a strict partition
$\lambda=(\lambda_1>\cdots>\lambda_\ell>0)$, set
\begin{equation}\label{eq:shifted-t-Schur-definition}
 \cQ_\lambda(X;t)
 =\Yop_{-\lambda_1}(t)\cdots
  \Yop_{-\lambda_\ell}(t)\vac.
\end{equation}

\subsection{Diagonal conjugation}

Define the graded algebra automorphism
\begin{equation}\label{eq:rho-definition}
 \rhoop{t}:\Gam_{\mathbb{Q}(t)}\longrightarrow
             \Gam_{\mathbb{Q}(t)},
 \qquad
 \rhoop{t}(p_n)=(1-t^n)p_n
 \quad(n\text{ odd}).
\end{equation}
Its inverse is given by
$\rhoop{t}^{-1}(p_n)=(1-t^n)^{-1}p_n$.

\begin{proposition}[Diagonal conjugation]\label{prop:conjugation}
The modified operator is a conjugate of the classical neutral operator:
\begin{equation}\label{eq:conjugation}
 \Yop(z;t)=\rhoop{t}\Phi(z)\rhoop{t}^{-1}.
\end{equation}
Consequently,
\begin{equation}\label{eq:plethystic-realization}
 \cQ_\lambda(X;t)
 =\rhoop{t}Q_\lambda(X)
 =Q_\lambda[X-tX].
\end{equation}
\end{proposition}

\begin{proof}
Because $a_{-n}$ is multiplication by $p_n$,
\[
 \rhoop{t}a_{-n}\rhoop{t}^{-1}
 =(1-t^n)a_{-n}.
\]
For $f\in\Gam_{\mathbb{Q}(t)}$,
\[
 \rhoop{t}
 \frac{\partial}{\partial p_n}
 \rhoop{t}^{-1}f
 =\frac{1}{1-t^n}
   \frac{\partial f}{\partial p_n}.
\]
Hence
\[
 \rhoop{t}a_n\rhoop{t}^{-1}
 =\frac{1}{1-t^n}a_n.
\]
Conjugating the two exponentials in \eqref{eq:Phi} gives exactly
\eqref{eq:Yt}, proving \eqref{eq:conjugation}.

Since $\rhoop{t}^{-1}\vac=\vac$, the mode expansion of
\eqref{eq:conjugation} gives
\[
 \begin{aligned}
 \cQ_\lambda(X;t)
 &=\rhoop{t}\phi_{-\lambda_1}\rhoop{t}^{-1}
   \cdots
   \rhoop{t}\phi_{-\lambda_\ell}\rhoop{t}^{-1}\vac\\
 &=\rhoop{t}
   \phi_{-\lambda_1}\cdots\phi_{-\lambda_\ell}\vac
 =\rhoop{t}Q_\lambda(X).
 \end{aligned}
\]
In plethystic notation,
$p_n[X-tX]=(1-t^n)p_n[X]$, which proves the final equality in
\eqref{eq:plethystic-realization}.
\end{proof}

\begin{corollary}\label{cor:power-sum-shifted}
For $\lambda\in\SP(N)$,
\begin{equation}\label{eq:power-sum-shifted}
 \cQ_\lambda(X;t)
 =\sum_{\alpha\in\OP(N)}
   \frac{2^{\ell(\alpha)}}{z_\alpha}
   \chi_\alpha^\lambda
   \prod_{a\in\alpha}(1-t^a)
   p_\alpha(X).
\end{equation}
In particular, $\cQ_\lambda(X;t)$ is polynomial in $t$ with
coefficients in $\Gam$.
\end{corollary}

\begin{proof}
Apply $\rhoop{t}$ to \eqref{eq:spin-expansion} and use
\[
 \rhoop{t}(p_\alpha)
 =\prod_{a\in\alpha}(1-t^a)p_\alpha.
\]
\end{proof}

\begin{remark}\label{rem:t-one}
The operator \eqref{eq:Yt} itself cannot be specialized directly at
$t=1$ because its second exponential contains $(1-t^n)^{-1}$. The
functions in \eqref{eq:power-sum-shifted}, however, are polynomial in
$t$. For every nonempty strict partition $\lambda$,
\[
 \cQ_\lambda(X;1)=0.
\]
Thus the specialization of the functions exists even though the
specialization of the operator does not.
\end{remark}

\begin{remark}\label{rem:Hopf}
Every odd power sum is primitive:
\[
 \Delta(p_n)=p_n\otimes1+1\otimes p_n.
\]
It follows that $\rhoop{t}$ is a graded Hopf algebra automorphism over
$\mathbb{Q}(t)$. In particular, same-parameter product and coproduct
identities for the $\cQ_\lambda(X;t)$ are transported directly from the
classical Schur $Q$ theory. The new part of the story appears when two
different parameters are compared.
\end{remark}

\section{Mixed operator products}
\label{sec:mixed}

We now work over $\K=\mathbb{Q}(t,s)$. Write
\begin{equation}\label{eq:A-B-definition}
 A_t(z)
 =\sum_{\substack{n\geq1\\n\text{ odd}}}
   \frac{2(1-t^n)}{n}a_{-n}z^n,
 \qquad
 B_t(z)
 =-\sum_{\substack{n\geq1\\n\text{ odd}}}
   \frac{2}{n(1-t^n)}a_nz^{-n}.
\end{equation}
Then
\[
 \Yop(z;t)=e^{A_t(z)}e^{B_t(z)}.
\]
We put the $A$-parts to the left and the $B$-parts to the right, and
write
\begin{equation}\label{eq:mixed-normal-order}
 :\Yop(z;t)\Yop(w;s):
 =e^{A_t(z)+A_s(w)}e^{B_t(z)+B_s(w)}.
\end{equation}
All scalar factors below are expanded as formal power series in
$w/z$.

\subsection{The mixed operator product expansion}

\begin{theorem}[Mixed operator product]\label{thm:mixed-OPE}
For independent parameters $t$ and $s$,
\begin{equation}\label{eq:mixed-OPE}
 \Yop(z;t)\Yop(w;s)
 =F_{t,s}(w/z):\Yop(z;t)\Yop(w;s):,
\end{equation}
where
\begin{equation}\label{eq:F-exponential}
 F_{t,s}(u)
 =\exp\!\left(
  -2\sum_{\substack{n\geq1\\n\text{ odd}}}
   \frac{1-s^n}{1-t^n}\frac{u^n}{n}
 \right)
 \in\K[[u]].
\end{equation}
When $s=t$,
\begin{equation}\label{eq:F-diagonal}
 F_{t,t}(u)=\frac{1-u}{1+u}.
\end{equation}
\end{theorem}

\begin{proof}
The $A$-parts commute with each other, and the $B$-parts also commute
with each other. By \eqref{eq:Heisenberg-commutator},
\[
 \begin{aligned}
 [B_t(z),A_s(w)]
 &=-\sum_{\substack{m,n\geq1\\m,n\text{ odd}}}
   \frac{2}{m(1-t^m)}
   \frac{2(1-s^n)}{n}
   [a_m,a_{-n}]z^{-m}w^n\\
 &=-2\sum_{\substack{n\geq1\\n\text{ odd}}}
   \frac{1-s^n}{1-t^n}\frac{(w/z)^n}{n}.
 \end{aligned}
\]
This commutator is a scalar. Therefore
\[
 e^{B_t(z)}e^{A_s(w)}
 =e^{[B_t(z),A_s(w)]}e^{A_s(w)}e^{B_t(z)}.
\]
Substituting this into
$e^{A_t(z)}e^{B_t(z)}e^{A_s(w)}e^{B_s(w)}$ proves
\eqref{eq:mixed-OPE} and \eqref{eq:F-exponential}.

If $s=t$, the ratio in \eqref{eq:F-exponential} is $1$. Since
\[
 2\sum_{\substack{n\geq1\\n\text{ odd}}}\frac{u^n}{n}
 =\log\frac{1+u}{1-u},
\]
we obtain \eqref{eq:F-diagonal}.
\end{proof}

\begin{corollary}[Several mixed operators]\label{cor:multiple-OPE}
For parameters $t_1,\ldots,t_r$,
\begin{equation}\label{eq:multiple-OPE}
 \Yop(z_1;t_1)\cdots\Yop(z_r;t_r)
 =\prod_{1\leq i<j\leq r}
  F_{t_i,t_j}(z_j/z_i)
  :\Yop(z_1;t_1)\cdots\Yop(z_r;t_r):.
\end{equation}
Each factor is expanded in nonnegative powers of $z_j/z_i$.
\end{corollary}

\begin{proof}
Move the $B$-part of the $i$th operator through the $A$-parts of the
operators to its right. Each ordered pair $i<j$ contributes exactly the
scalar factor computed in \cref{thm:mixed-OPE}.
\end{proof}

\subsection{Infinite-product form and coefficient recurrence}

For $m\geq0$, let
\[
 (a;t)_m=\prod_{j=0}^{m-1}(1-at^j),
 \qquad
 (a;t)_\infty=\prod_{j\geq0}(1-at^j).
\]

\begin{proposition}\label{prop:F-Pochhammer}
Analytically for $|t|<1$, or formally in the $t$-adic completion,
\begin{equation}\label{eq:F-Pochhammer}
 F_{t,s}(u)
 =\frac{(u;t)_\infty(-su;t)_\infty}
        {(-u;t)_\infty(su;t)_\infty}.
\end{equation}
\end{proposition}

\begin{proof}
Use
\[
 \frac{1-s^n}{1-t^n}
 =\sum_{j\geq0}t^{jn}
  -\sum_{j\geq0}(st^j)^n
\]
in \eqref{eq:F-exponential}. For each variable $x$,
\[
 \exp\!\left(
  -2\sum_{n\text{ odd}}\frac{x^n}{n}
 \right)
 =\frac{1-x}{1+x}.
\]
It follows that
\[
 F_{t,s}(u)
 =\prod_{j\geq0}
  \frac{(1-t^ju)(1+st^ju)}
       {(1+t^ju)(1-st^ju)},
\]
which is \eqref{eq:F-Pochhammer}.
\end{proof}

Write
\begin{equation}\label{eq:f-coefficients}
 F_{t,s}(u)=\sum_{m\geq0}f_m(t,s)u^m,
 \qquad f_0(t,s)=1.
\end{equation}

\begin{proposition}[Coefficient recurrence]\label{prop:f-recurrence}
For $m\geq1$,
\begin{equation}\label{eq:f-recurrence}
 m f_m(t,s)
 =-2\sum_{\substack{1\leq j\leq m\\j\text{ odd}}}
   \frac{1-s^j}{1-t^j}
   f_{m-j}(t,s).
\end{equation}
\end{proposition}

\begin{proof}
Logarithmically differentiate \eqref{eq:F-exponential}:
\[
 uF'_{t,s}(u)
 =-2F_{t,s}(u)
   \sum_{j\text{ odd}}
    \frac{1-s^j}{1-t^j}u^j.
\]
Comparing the coefficient of $u^m$ gives
\eqref{eq:f-recurrence}.
\end{proof}

\subsection{Coefficients after applying the operators to \texorpdfstring{$\vac$}{1}}

Define the modified one-row series
\begin{equation}\label{eq:Ht-definition}
 \Ht_t(X;z)
 :=\Yop(z;t)\vac
 =\exp\!\left(
   2\sum_{\substack{n\geq1\\n\text{ odd}}}
    \frac{1-t^n}{n}p_n(X)z^n
  \right)
 =\sum_{r\geq0}q_r(X;t)z^r.
\end{equation}
Thus $q_r(X;t)=\Yop_{-r}(t)\vac$.

For $r,m\geq0$, define the mixed two-row coefficient
\begin{equation}\label{eq:mixed-two-row-definition}
 \cQ_{r,m}^{t,s}(X)
 :=[z^rw^m]\Yop(z;t)\Yop(w;s)\vac
 =\Yop_{-r}(t)\Yop_{-m}(s)\vac.
\end{equation}
When $t\neq s$, these coefficients are not skew-symmetric in $r$ and
$m$, so we do not regard them as a strict-partition basis.

\begin{proposition}\label{prop:mixed-two-row}
For $r,m\geq0$,
\begin{equation}\label{eq:mixed-two-row}
 \cQ_{r,m}^{t,s}(X)
 =\sum_{k=0}^{m}
  f_k(t,s)q_{r+k}(X;t)q_{m-k}(X;s).
\end{equation}
For $s=t$, this becomes
\begin{equation}\label{eq:two-row-recovered}
 \cQ_{r,m}^{t,t}(X)
 =q_r(X;t)q_m(X;t)
  +2\sum_{k=1}^{m}(-1)^k
    q_{r+k}(X;t)q_{m-k}(X;t).
\end{equation}
\end{proposition}

\begin{proof}
Applying \eqref{eq:mixed-OPE} to $\vac$ gives
\[
 \Yop(z;t)\Yop(w;s)\vac
 =F_{t,s}(w/z)\Ht_t(X;z)\Ht_s(X;w).
\]
Substitute \eqref{eq:f-coefficients} and
\eqref{eq:Ht-definition}. To obtain $z^rw^m$, the three indices must
satisfy
\[
 z^{-k}z^{r+k}=z^r,
 \qquad
 w^kw^{m-k}=w^m,
\]
which gives \eqref{eq:mixed-two-row}. If $s=t$, then
\[
 \frac{1-u}{1+u}=1+2\sum_{k\geq1}(-1)^ku^k,
\]
so \eqref{eq:two-row-recovered} follows.
\end{proof}

\section{Cyclotomic mixed products}
\label{sec:cyclotomic-mixed}

Fix a positive integer $M$ and define the finite plethystic alphabet
\begin{equation}\label{eq:AM-definition}
 A_M(t)=1+t+\cdots+t^{M-1}.
\end{equation}
We call a polynomial $f(t)$ \emph{palindromic of darga $d$} if
$f(t)=t^d f(t^{-1})$; this convention permits zero coefficients at
both ends of the ambient degree interval $[0,d]$.
For every $n\geq1$,
\begin{equation}\label{eq:power-AM}
 p_n[A_M(t)]
 =1+t^n+\cdots+t^{(M-1)n}
 =\frac{1-t^{Mn}}{1-t^n}.
\end{equation}

\subsection{Finite mixed scalar factors}

\begin{theorem}[Cyclotomic mixed product]\label{thm:cyclotomic-OPE}
For $M\geq1$,
\begin{equation}\label{eq:cyclotomic-OPE}
 \Yop(z;t)\Yop(w;t^M)
 =\frac{(w/z;t)_M}{(-w/z;t)_M}
  :\Yop(z;t)\Yop(w;t^M):.
\end{equation}
Equivalently,
\begin{equation}\label{eq:F-cyclotomic}
 F_{t,t^M}(u)
 =\prod_{j=0}^{M-1}\frac{1-t^ju}{1+t^ju}
 =\frac{(u;t)_M}{(-u;t)_M}.
\end{equation}
\end{theorem}

\begin{proof}
By \eqref{eq:F-exponential} and \eqref{eq:power-AM},
\[
 \begin{aligned}
 \log F_{t,t^M}(u)
 &=-2\sum_{n\text{ odd}}
    \left(\sum_{j=0}^{M-1}t^{jn}\right)
    \frac{u^n}{n}\\
 &=\sum_{j=0}^{M-1}
   \log\frac{1-t^ju}{1+t^ju}.
 \end{aligned}
\]
Exponentiating gives \eqref{eq:F-cyclotomic}, and
\eqref{eq:cyclotomic-OPE} follows from
\cref{thm:mixed-OPE}.
\end{proof}

Write
\begin{equation}\label{eq:fM-definition}
 F_{t,t^M}(u)=\sum_{m\geq0}f_m^{[M]}(t)u^m.
\end{equation}

\begin{proposition}[Principal specialization]\label{prop:fM-principal}
For every $m\geq0$,
\begin{equation}\label{eq:fM-Q-principal}
 f_m^{[M]}(t)
 =(-1)^m q_m(1,t,\ldots,t^{M-1}).
\end{equation}
In particular,
\begin{equation}\label{eq:fM-sign-positive}
 (-1)^m f_m^{[M]}(t)\in\mathbb{N}[t].
\end{equation}
Moreover,
\begin{equation}\label{eq:fM-explicit}
 (-1)^m f_m^{[M]}(t)
 =\sum_{a=0}^{\min(M,m)}
   t^{\binom{a}{2}}
   \qbinom{M}{a}{t}
   \qbinom{M+m-a-1}{m-a}{t}.
\end{equation}
\end{proposition}

\begin{proof}
By \eqref{eq:one-row-generating},
\[
 \sum_{m\geq0}q_m(1,t,\ldots,t^{M-1})z^m
 =\prod_{j=0}^{M-1}\frac{1+t^jz}{1-t^jz}.
\]
Substitute $z=-u$ and compare with
\eqref{eq:F-cyclotomic}; this proves
\eqref{eq:fM-Q-principal} and \eqref{eq:fM-sign-positive}.

For the explicit formula, use the finite and reciprocal
$q$-binomial expansions
\[
 \prod_{j=0}^{M-1}(1+t^ju)
 =\sum_{a=0}^{M}
   t^{\binom{a}{2}}\qbinom{M}{a}{t}u^a
\]
and
\[
 \prod_{j=0}^{M-1}\frac{1}{1-t^ju}
 =\sum_{b\geq0}
   \qbinom{M+b-1}{b}{t}u^b.
\]
The coefficient of $u^m$ in their product is the right side of
\eqref{eq:fM-explicit}.
\end{proof}

\begin{corollary}[Reciprocity for scalar factors]\label{cor:fM-reciprocity}
For every $m\geq0$,
\begin{equation}\label{eq:fM-reciprocity}
 f_m^{[M]}(t)
 =t^{(M-1)m}f_m^{[M]}(t^{-1}).
\end{equation}
Thus $(-1)^mf_m^{[M]}(t)$ is palindromic with darga
$(M-1)m$.
\end{corollary}

\begin{proof}
As multisets,
\[
 A_M(t)=t^{M-1}A_M(t^{-1}).
\]
Since $q_m$ is homogeneous of degree $m$,
\[
 q_m[A_M(t)]
 =t^{(M-1)m}q_m[A_M(t^{-1})].
\]
Apply \eqref{eq:fM-Q-principal}.
\end{proof}

\begin{proposition}[Finite-order recurrence]\label{prop:finite-recurrence}
Set
\begin{equation}\label{eq:eMj-definition}
 e_j^{[M]}(t)
 :=e_j(1,t,\ldots,t^{M-1})
 =t^{\binom{j}{2}}\qbinom{M}{j}{t},
 \qquad 0\leq j\leq M.
\end{equation}
Then the coefficients in \eqref{eq:fM-definition} satisfy
\begin{equation}\label{eq:finite-recurrence}
 \sum_{j=0}^{\min(M,m)}
 e_j^{[M]}(t)f_{m-j}^{[M]}(t)
 =
 \begin{cases}
  (-1)^m e_m^{[M]}(t),&0\leq m\leq M,\\
  0,&m>M.
 \end{cases}
\end{equation}
In particular, for $m>M$,
\begin{equation}\label{eq:homogeneous-finite-recurrence}
 f_m^{[M]}(t)
 =-\sum_{j=1}^{M}e_j^{[M]}(t)f_{m-j}^{[M]}(t).
\end{equation}
\end{proposition}

\begin{proof}
The numerator and denominator in \eqref{eq:F-cyclotomic} expand as
\[
 (u;t)_M
 =\sum_{j=0}^{M}(-1)^j e_j^{[M]}(t)u^j,
 \qquad
 (-u;t)_M
 =\sum_{j=0}^{M}e_j^{[M]}(t)u^j.
\]
Multiplying
\[
 (-u;t)_M\sum_{m\geq0}f_m^{[M]}(t)u^m=(u;t)_M
\]
and comparing the coefficient of $u^m$ gives
\eqref{eq:finite-recurrence}. The homogeneous recurrence follows when
$m>M$.
\end{proof}

\begin{corollary}[Cyclotomic mixed coefficients]
\label{cor:cyclotomic-vacuum}
For $r,m\geq0$,
\begin{equation}\label{eq:cyclotomic-vacuum}
 \cQ_{r,m}^{t,t^M}(X)
 =\sum_{k=0}^{m}(-1)^k
  q_k(1,t,\ldots,t^{M-1})
  q_{r+k}(X;t)q_{m-k}(X;t^M).
\end{equation}
\end{corollary}

\begin{proof}
Combine \cref{prop:mixed-two-row} with
\eqref{eq:fM-Q-principal}.
\end{proof}

\begin{example}[The case $M=2$]\label{ex:M2-contraction}
For $M=2$,
\[
 F_{t,t^2}(u)
 =\frac{(1-u)(1-tu)}{(1+u)(1+tu)}.
\]
Here
\[
 e_1^{[2]}(t)=1+t,
 \qquad
 e_2^{[2]}(t)=t,
\]
and
\begin{align*}
 f_0^{[2]}(t)&=1,\\
 f_1^{[2]}(t)&=-2(1+t),\\
 f_2^{[2]}(t)&=2(1+t)^2,\\
 f_3^{[2]}(t)&=-2(1+t)(1+t+t^2).
\end{align*}
For $m>2$, the coefficients satisfy
\[
 f_m^{[2]}(t)
 =-(1+t)f_{m-1}^{[2]}(t)-t f_{m-2}^{[2]}(t).
\]
\end{example}

\end{document}